\title{Characterizing Biautomatic Groups}
\date{March 2021}
\author {Aischa Amrhein}

\documentclass[a4paper,11pt,twoside]{article}
\usepackage[a4paper,left=2.5cm, right=2.5cm,top=2.5cm, bottom=2.6cm]{geometry}
\usepackage{mathrsfs}

\usepackage{amsmath}
\usepackage{amsfonts}
\usepackage{amssymb}
\usepackage{amsthm}
\usepackage[all]{xy}
\usepackage{setspace}
\usepackage{tikz-cd}
\usepackage{tikz}

\usepackage{enumitem}
\usepackage{xfrac}
\usepackage{faktor}
\usepackage{mathtools}
\usepackage{bbold}

\theoremstyle{definition}
\newtheorem{de}{Definition}
\newtheorem{thm}[de]{Theorem}

\newtheorem{lem}[de]{Lemma}

\newtheorem{ex}[de]{Example}

\theoremstyle{remark}
\newtheorem*{con}{Convention}

\usepackage[nottoc]{tocbibind}

\newcommand{\Z}{\mathbin{Z}}

\begin{document}
\maketitle
This paper corrects the characterisation of biautomatic groups presented in Lemma 2.5.5 in the book \textit{Word Processing in Groups} by Epstein et al.  \cite{wordproc}. We present a counterexample to the lemma, and we reformulate the lemma to give a valid characterisation of biautomatic groups. \newline

We use the same notation and core definitions for automatic and biautomatic groups as the book, most importantly we use the following conventions.
\begin{con}
Let $A$ be an alphabet generating a group $G$ as a semi-group and let $w \in A^*$ be a word. Then $\pi(w) \in G$ is the group element represented by the word $w$, and $\widehat{w}:[0, \infty) \to \Gamma(G,A)$ is the path in the Cayley graph that travels at speed one from the neutral element to $\pi(w)$, where it stays for good. The (synchronous, uniform) distance between two paths $\widehat{w_1}$ and $\widehat{w_2}$ is given by the maximum of the point-wise distance at any point in time.
\end{con}

The following definitions correspond to Theorem 2.3.5 and Definition 2.5.4 in \cite{wordproc}.
\begin{de}[Automatic Structure]
\label{de: automatic}
Let $A$ be a finite set generating the group $G$ as a semigroup, and let $L$ be a regular language over $A$ that maps onto $G$. Then $(A,L)$ is an automatic structure if there exists some constant $k$ such that for all  $w_1,w_2 \in L$ and for all $a \in A \cup \{\varepsilon\}$ fulfilling $\pi(w_1a) = \pi(w_2)$ it holds that the distance between $\widehat{w_1a}$ and $\widehat{w_2}$ is less than or equal to $k$.
\end{de}
\begin{de}[Biautomatic Structure]
\label{de: biautomatic}
Let $G$ be a group, $A$ an alphabet closed under inversion and $L \subset A^*$ a language that maps onto $G$. Then $(A,L)$ is a biautomatic structure if both $(A,L)$ and $(A,L^{-1})$ are automatic structures.
\end{de}

The following statement is Lemma 2.5.5. in \cite{wordproc}, it aims to characterize biautomatic groups.
\begin{lem}
\label{lem: wrong}
Let $G$ be a group and let $A$ be a finite set of semi-group generators closed under inversion. Let $L$ be a regular language over $A$ that maps onto $G$. Then $L$ is biautomatic if and only if, for each $w_1,w_2 \in L$ and for each $a,b \in A \cup \{\varepsilon\}$ fulfilling $\pi(aw_1b) = \pi(w_2)$ it holds that the maximum distance between the paths $\widehat{aw_1b}$ and $\widehat{w_2}$ is bounded by some fixed constant $k$. 
\end{lem}

The book \cite{wordproc} provides a proof sketch for this lemma: first it verifies that $L^{-1}$ is regular if $L$ is, then it claims that the result follows from Definition \ref{de: automatic} (Automatic Structure). However, a difficulty arises when one tries to bound the distance between two inverted paths in the Cayley graph, as they may be badly synchronized. The following is a counterexample to Lemma \ref{lem: wrong}.

\begin{ex}
Let $G = \Z^2$ and consider the structure $(A,L)$ where $A = \{x,y,x^{-1},y^{-1}\}$ and $L = \{x^my^n | m,n \in \Z\}(xyx^{-1}y^{-1})^*$. We show that this structure is automatic but not biautomatic. \\
\begin{figure}[ht]
    \centering
    \includegraphics[width = 0.75\textwidth]{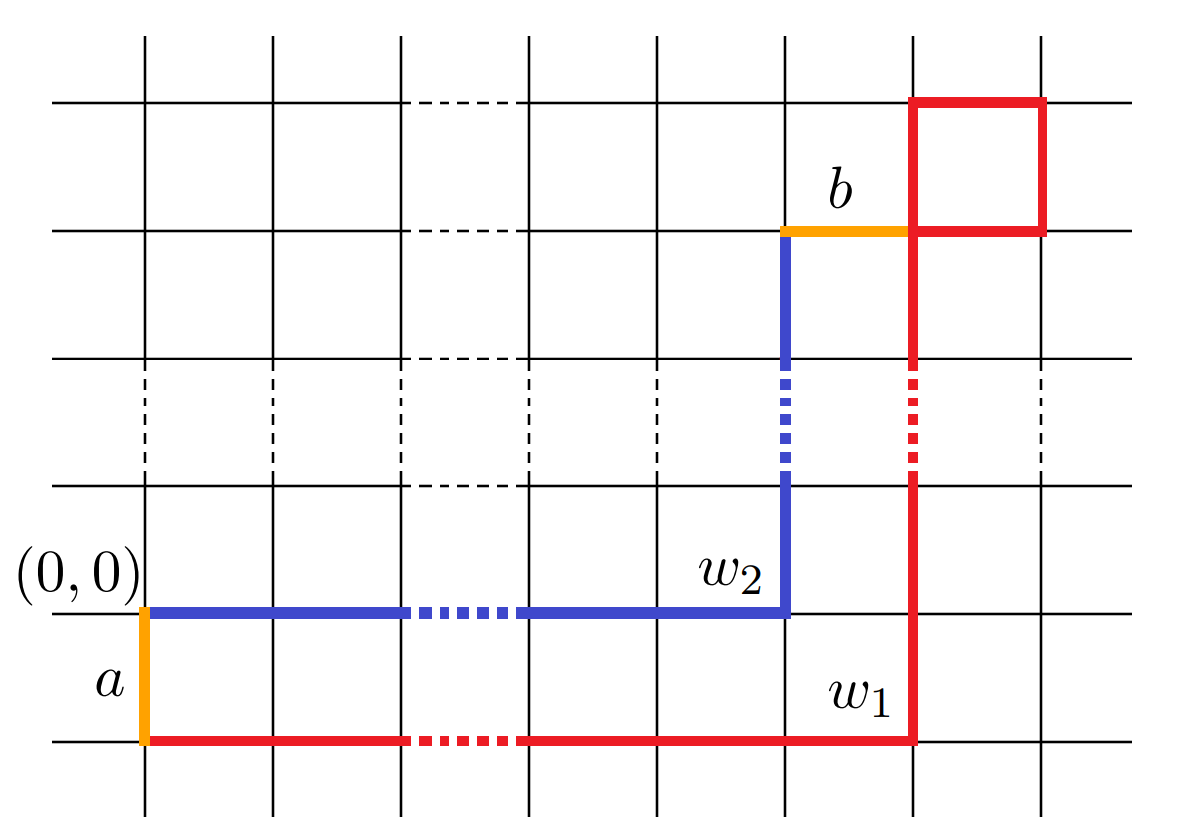}
   \caption{The blue path is generated by $w_2 = x^m y^n$, the orange/red path by $a w_1 b$ where $a = y^{-1}$ (in orange), $w_1 = x^{m+1}y^{n+1}(xyx^{-1}y^{-1})^k$ (in red), and $b = x^{-1}$ (in orange). They fulfill $\pi(aw_1b) = \pi(w_2)$.}
    \label{fig: Counterexample}
\end{figure}

For some word $w \in L$ the path $\widehat{w}$ looks as follows: first it travels right or left, then up or down, and finally, it performs an amount of loops with the basepoint where it ends. This language does not fulfill the uniqueness property. However, $\pi(w)$ determines the first part of the path, the only variation left is the amount of loops the path performs. A loop $xyx^{-1}y^{-1}$ has a uniform distance of $2$ to its base point. This implies that for $w_1,w_2 \in L$ and $a \in A \cup \{\varepsilon\}$ fulfilling $\pi(w_1a) = \pi(w_2)$ the uniform distance between $\widehat{w_1a}$ and $\widehat{w_2}$ is at most $3$. Therefore, $(A,L)$ is an automatic structure. \\
The inverse of the language $L$ is given by $L^{-1} = (yxy^{-1}x^{-1})^*L$. Let $w_1 = (yxy^{-1}x^{-1})^nx^n$ and $w_2 = x^n$. Clearly, $\pi(w_1) = \pi(w_2)$. However, the maximum distance between the two paths is greater than $n$, thus $L$ is not biautomatic. \\
We claim that this structure fulfills the characterization in Lemma \ref{lem: wrong}. Let $a,b \in A \cup \{\varepsilon\}$ and $w_1,w_2 \in L$ be such that $\pi(aw_1b) = \pi(w_2)$. Then the distance between the paths $\widehat{aw_1b}$ and $\widehat{w_2}$ is at most $4$, by the same argument that we used when showing that $(A,L)$ is automatic, and using that $G$ is abelian. Figure \ref{fig: Counterexample} shows the situation.
\end{ex}
A simpler example (though visually less representable), is given by $G = \Z$ and $L = \{x^n|n \in \Z\}(xx^{-1})^*$. Again, the paths vary at their endpoint over a bounded distance but during an unbounded amount of time, leaving the reversal of the paths an unbounded distance apart. \\

It seems that after the appearance of the preprint version of \cite{wordproc}, the two-sided fellow traveller condition occurring in Lemma \ref{lem: wrong} was soon adopted as the standard definition of biautomaticity (see \cite{GS1} and \cite{GS2}) and was subsequently used in most publications pertaining to the subject. The question now arises whether the original definition should be regarded as non-standard, or out-dated.

On the other hand, as was pointed out to the author by Ian Leary and Ashot Minasyan, the characterization in Lemma~3 works with the additional condition that the map $\pi \colon L \to G$ is finite-to-one. This condition is not restrictive and occurs also in other contexts; see, for example, Theorem 2.5.1 and Theorem 3.3.4 in \cite{wordproc}. This modified characterization relies on the following lemma, which resembles Lemma 2.3.9 in the book.

\begin{lem}
\label{lem: bounded length difference}
Let $A$ be a finite set generating the group $G$ as a semigroup, and let $L$ be a regular language over $A$ such that every group element is represented by at most finitely many words in $L$. Then for every $k$ there is a constant $N$ such that whenever $a \in A \cup \{\epsilon\}$ and $w_1,w_2 \in L$ are such that the distance between the paths $\widehat{aw_1}$ and $\widehat{w_2}$ is at most $k$, the lengths of $w_1$ and $w_2$ differ by at most $N$.
\end{lem}
\begin{proof}
Let $l_1$ and $l_2$ be the lengths of $aw_1$ and $w_2$, respectively, and let $n$ be the number of states in an automaton accepting $L$. Suppose first that $l_1 < l_2$. The path $\widehat{w_2}|_{[l_1,l_2]}$ remains in the closed ball of radius $k$ about the point $\widehat{aw_1}(l_1)$. If this path visits a vertex in the ball more than $n$ times, then there exist two distinct integers $t,t' \in [l_1,l_2]$ such that $\widehat{w_2}(t) = \widehat{w_2}(t')$ and the automaton is at the same state after reading the prefixes $w_2(t)$ and $w_2(t')$. By going over the loop $\widehat{w_2}|_{[t,t']}$ repeatedly, one would get arbitrarily many representatives of the same group element. It follows that $l_2 - l_1$ is less than $n$ times the number of vertices in the $k$-ball. The case $l_2 < l_1$ is analogous.
\end{proof}
Using the lemma, we can give a characterisation of biautomatic groups.
\begin{thm}
\label{thm: characterizing finite-to-one biautomatic}
Let $G$ be a group and let $A$ be a finite set of generators closed under inversion.
Let $L$ be a regular language over $A$ such that $\pi:L \to G$ is finite-to-one and onto. Then $(A,L)$ is a biautomatic structure if and only if there exists some $k$ such that for all $w_1,w_2 \in L$ and for all $a,b \in A\cup \{\varepsilon\}$ fulfilling $\pi(aw_1b) = \pi(w_2)$ it holds that the distance between the paths $\widehat{aw_1b}$ and $\widehat{w_2}$ is less than or equal to $k$.
\end{thm}
\begin{proof}
First we show that a structure $(A,L)$ satisfying the criterion in the theorem fulfills the conditions of Definition \ref{de: biautomatic}. Let $w_1,w_2 \in L$ and $a \in A\cup\{\varepsilon\}$ fulfill $\pi(w_1a) = \pi(w_2)$. By assumption, this means that the distance between $\widehat{w_1a}$ and $\widehat{w_2}$ is less than or equal to $k$, thus $(A,L)$ is an automatic structure. Now assume that $\pi(w_1^{-1}a) = \pi(w_2^{-1})$, or equivalently $\pi(a^{-1}w_1) = \pi(w_2)$. By assumption, the distance between the paths $\widehat{a^{-1}w_1}$ and $\widehat{w_2}$ is bounded by some $k$. Using Lemma \ref{lem: bounded length difference}, we obtain that the length difference between $a^{-1}w_1$ and $w_2$ is bounded by a constant $N$. Reversing the paths gives that the distance between $\widehat{w_1^{-1}a}$ and $\widehat{w_2^{-1}}$ is bounded by $N+k$. Since $L^{-1}$ is also regular, $(A,L^{-1})$ is an automatic structure. \newline
Conversely, assume $(A,L)$ is a biautomatic structure and we have $w_1,w_2 \in L$ with $\pi(aw_1b) = \pi(w_2)$ for some $a,b \in A\cup\{\varepsilon\}$. Let $u \in L$ be such that $\pi(u) = \pi(aw_1)$, so $\pi(ub) = \pi(w_2)$. By assumption, the distance between $\widehat{u}$ and $\widehat{w_2}$ is bounded by some $k$, the same holds if we add $b$ to the first path. It holds that $\pi(u^{-1}) = \pi(w_1^{-1}a^{-1})$, therefore, by the same argument as before, the distance between the paths $\widehat{u^{-1}}$ and $\widehat{w_1^{-1}a^{-1}}$ is bounded by $k$. By Lemma \ref{lem: bounded length difference}, the difference in length between the paths is bounded by some $N$. Reversing the paths and accounting for the difference in length gives that the paths $\widehat{u}$ and $\widehat{aw_1}$ are at most $k+N$ apart. In conclusion, the paths $\widehat{aw_1b}$ and $\widehat{w_2}$ have a maximum distance of at most $N+2k+1$. 
\end{proof}

This work is an excerpt of \cite{introd}.
I would like to extend my sincere gratitude to Professor Urs Lang for his encouragement and his support in writing this note.

\bibliographystyle{abbrv}
\bibliography{sample}

\end{document}